\documentclass[AMA,STIX1COL]{WileyNJD-v2}
\usepackage{moreverb}

\newcommand\BibTeX{{\rmfamily B\kern-.05em \textsc{i\kern-.025em b}\kern-.08em
T\kern-.1667em\lower.7ex\hbox{E}\kern-.125emX}}

\articletype{Article Type}%

\received{<day> <Month>, <year>}
\revised{<day> <Month>, <year>}
\accepted{<day> <Month>, <year>}

\newcommand{\La}{\mathcal{L}}

\newtheorem{ex}{Example}

\begin{document}

\title{Stability of two-component incommensurate fractional-order systems and applications to the investigation of a FitzHugh-Nagumo neuronal model \protect\thanks{This work was supported by a grant of the Romanian National Authority for
Scientific Research and Innovation, CNCS-UEFISCDI, project no. PN-II-RU-TE-2014-4-0270.}}

\author[1,2]{Oana Brandibur}
\author[1,2]{Eva Kaslik*}

\authormark{OANA BRANDIBUR \& EVA KASLIK}

\address[1]{\orgdiv{Dept. of Math. and Comp. Sci.}, \orgname{West University of Timi\c{s}oara}, \country{Romania}}

\address[2]{\orgname{Institute e-Austria Timi\c{s}oara}, \country{Romania}}

\corres{*Eva Kaslik, \email{ekaslik@gmail.com}}

%\presentaddress{Present address}

\abstract[Abstract]{
For two-dimensional autonomous linear incommensurate fractional-order dynamical systems with Caputo derivatives of different orders, necessary and sufficient conditions are obtained for the asymptotic stability and instability of the null solution. These conditions are expressed in terms of the elements of the system's matrix, as well as of the fractional orders of the Caputo derivatives, leading to a generalization of the well known Routh-Hurwitz conditions. These theoretical results are then used to investigate the stability properties of a two-dimensional fractional-order FitzHugh-Nagumo neuronal model. The occurrence of Hopf bifurcations is also discussed. Numerical simulations are provided with the aim of exemplifying the theoretical results, revealing rich spiking behavior, in comparison with the classical integer-order FitzHugh-Nagumo model.}


%\jnlcitation{\cname{%
%\author{Williams K.},
%\author{B. Hoskins},
%\author{R. Lee},
%\author{G. Masato}, and
%\author{T. Woollings}} (\cyear{2016}),
%\ctitle{A regime analysis of Atlantic winter jet variability applied to evaluate HadGEM3-GC2}, \cjournal{Q.J.R. Meteorol. Soc.}, \cvol{2017;00:1--6}.}

\maketitle

%  Main text of the article.

\section{Introduction}

In many real world applications \cite{Cottone,Engheia,Henry_Wearne,Heymans_Bauwens,Mainardi_1996}, fractional-order dynamical systems have proven to provide more accurate and realistic results than their classical integer-order counterparts, due to the fact that fractional-order derivatives, as non-local operators, are able to reflect memory and hereditary properties. However, it is important to emphasize that important qualitative differences may appear when generalizing properties of integer-order dynamical systems to the fractional-order case, and such generalizations have to be done with great care.

Stability analysis is one of the most important research topics of the qualitative theory of fractional-order systems. Two recent surveys \cite{Li-survey,Rivero2013stability} provide comprehensive overviews of stability properties of fractional-order systems. In the particular case of linear autonomous commensurate fractional order systems, the most important starting point is Matignon's stability theorem \cite{Matignon}, which has been generalized in \cite{Sabatier2012stability}. Linearization theorems (or analogues of the classical Hartman-Grobman theorem) for fractional-order systems have been recently proved in \cite{Li_Ma_2013,Wang2016stability}. Up to this date, incommensurate order systems have not received as much attention as their commensurate order counterparts.  Linear incommensurate fractional order systems with rational orders have been analyzed in \cite{Petras2008stability}. Oscillations in two-dimensional incommensurate fractional order systems have been investigated in \cite{Datsko2012complex,Radwan2008fractional}. BIBO stability of systems with irrational transfer functions has been recently investigated in \cite{trachtler2016bibo}.

The first aim of this paper is to explore necessary and sufficient conditions for the asymptotic stability of two-dimensional linear autonomous incommensurate fractional-order systems with Caputo derivatives of different orders. An extended summary of these results has been given in \cite{Brandibur_CMMSE2017}. The theoretical results included in the following sections generalize the previous findings reported in \cite{Brandibur_2016}, concerning two-dimensional systems composed of a fractional-order differential equation and a classical first-order differential equation. These results are later applied to investigate stability properties of a fractional-order FitzHugh-Nagumo neuronal model. It is worth emphasizing that the fractional-order formulation of neuronal dynamics is strongly justified by experimental results concerning biological neurons \cite{Anastasio,Lundstrom}. Moreover, \cite{du2013measuring} suggests that a possible physical meaning of the order of a fractional derivative is that of an index of memory, which further justifies its use in mathematical models arising from neuroscience.

\section{Preliminaries}

Consider the  $n$-dimensional fractional-order system
with Caputo derivatives \begin{equation}\label{sys.gen}
^c\!D^\mathbf{q}\mathbf{x}(t)=f(t,\mathbf{x})
\end{equation}
where $\mathbf{q}=(q_1,q_2,...,q_n)\in(0,1)^n$ and $f:[0,\infty)\times\mathbb{R}^n\rightarrow \mathbb{R}^n$ is continuous on the whole domain of definition and Lipschitz-continuous with respect to the second variable, such that
$$f(t,0)=0\quad \textrm{for any }t\geq 0.$$
Let $\varphi(t,x_0)$ denote the unique solution of (\ref{sys.gen}) which satisfies the initial condition $x(0)=x_0\in\mathbb{R}^n$. The existence and uniqueness of the initial value problem associated to system (\ref{sys.gen}) is guaranteed by the properties of the function $f$ stated above \cite{Diethelm_book}.

Generally, the asymptotic stability of the trivial solution of system (\ref{sys.gen}) is not of exponential type \cite{Cermak2015,Gorenflo_Mainardi}, because of the presence of the memory effect. A special type of non-exponential asymptotic stability concept has been defined for fractional-order differential equations \cite{Li_Chen_Podlubny}, called Mittag-Leffler stability.  In this paper, we are concerned with  $\mathcal{O}(t^{-\alpha})$-asymptotic stability, which reflects the algebraic decay of the solutions.

\begin{definition}\label{def.stability}
The trivial solution of (\ref{sys.gen}) is called \emph{stable} if for any $\varepsilon>0$
there exists $\delta=\delta(\varepsilon)>0$ such that for every $x_0\in\mathbb{R}^n$ satisfying $\|x_0\|<\delta$ we have
$\|\varphi(t,x_0)\|\leq\varepsilon$ for any $t\geq 0$.

The trivial solution  of (\ref{sys.gen}) is called \emph{asymptotically stable} if it is stable and there
exists $\rho>0$ such that $\lim\limits_{t\rightarrow\infty}\varphi(t,x_0)=0$ whenever $\|x_0\|<\rho$.

Let $\alpha>0$. The trivial solution  of (\ref{sys.gen}) is called \emph{$\mathcal{O}(t^{-\alpha})$-asymptotically stable} if it is stable and there exists $\rho>0$ such that for any $\|x_0\|<\rho$ one has:
$$\|\varphi(t,x_0)\|=\mathcal{O}(t^{-\alpha})\quad\textrm{as }t\rightarrow\infty.$$
\end{definition}

\section{Stability results for linear systems with two Caputo derivatives of different orders}

The following two-dimensional linear autonomous incommensurate fractional-order system is considered:
\begin{equation}\label{linearsys}
\left\{
\begin{array}{l}
  ^c\!D^{q_1}x(t)=a_{11}x(t)+a_{12}y(t) \\
  ^c\!D^{q_2}x(t)=a_{21}x(t)+a_{22}y(t)
\end{array}
\right.
\end{equation}
where $A=(a_{ij})$ is a real 2-dimensional matrix and $q_1,q_2\in(0,1)$ are the fractional orders of the Caputo derivatives. We may assume $a_{12}a_{21}\neq 0$, because otherwise, one equation would be decoupled.

Applying the Laplace transform to system (\ref{linearsys}) we obtain the following system:
\begin{equation*}
\begin{bmatrix}
s^{q_1}X(s)-s^{q_1-1}x(0)\\s^{q_2}Y(s)-s^{q_2-1}y(0)
\end{bmatrix}=A\cdot \begin{bmatrix}
X(s)\\Y(s)
\end{bmatrix},
\end{equation*}
where $X(s)=\La(x)(s)$ and $Y(s)=\La(y)(s)$ represent the Laplace transforms of the functions $x$ and $y$, respectively, and $s^{q_1}$, $s^{q_2}$ represent the principal values (first branches) of the corresponding complex power functions \cite{Doetsch}. Therefore:
$$\left(\text{diag}(s^{q_1},s^{q_2})-A\right)\cdot\begin{bmatrix}
X(s)\\Y(s)
\end{bmatrix}=\begin{bmatrix}
s^{q_1-1} x(0)\\ s^{q_2-1}y(0)
\end{bmatrix}.$$
In the following, we denote
\begin{equation*}
\Delta_A(s)=\det\left(\text{diag}(s^{q_1},s^{q_2})-A\right)=s^{q_1+q_2}-a_{11}s^{q_2}-a_{22}s^{q_1}+\det(A)
\end{equation*}
and therefore, we obtain
\begin{equation}\label{transfer.functions}
X(s)=\frac{s^{q_1}(s^{q_2}-a_{22})x(0)+a_{12}s^{q_2}y(0)}{s\Delta_A(s)}\quad\text{and}\quad	Y(s)=\frac{s^{q_2}(s^{q_1}-a_{11})y(0)+a_{21}s^{q_1}x(0)}{s\Delta_A(s)}
\end{equation}

The following result provides necessary and sufficient conditions for the global asymptotic stability of system (\ref{linearsys}) as well as sufficient conditions for the instability of  system (\ref{linearsys}). The proof is based on the Final Value Theorem and asymptotic expansion properties of the Laplace transform \cite{Doetsch,Bonnet_2000,Brandibur_2016}.

\begin{theorem}\label{thm.lin.stab}$ $
\begin{enumerate}
\item Denoting $q=\min\{q_1,q_2\}$, system (\ref{linearsys}) is $\mathcal{O}(t^{-q})$-globally asymptotically stable if and only if all the roots of $\Delta_A(s)$ are in the open left half-plane ($\Re(s)<0$).
\item If $\det(A)\neq 0$ and $\Delta_A(s)$ has a root in the open right half-place ($\Re(s)>0$), system (\ref{linearsys}) is unstable.
\end{enumerate}
\end{theorem}

\begin{proof}
Without loss of generality, we consider $0<q_1\leq q_2<1$.

\emph{Part 1 - Necessity.} 
	Assuming that (\ref{linearsys}) is $\mathcal{O}(t^{-q})$-globally asymptotically stable and letting $(x(t),y(t))$ denote the solution of system (\ref{linearsys}) which satisfies the initial condition $(x(0),y(0))=(x_0,y_0)\in\mathbb{R}^2$ such that $x_0,y_0\ne 0$, there exist $M>0$ and $T>0$ such that
	$$|x(t)|\leq\|(x(t),y(t))\|\leq Mt^{-q},\ \forall \ t\geq T.$$
Hence, the Laplace transform  
$$X(s)=\frac{x_0s^{q_1}(s^{q_2}-a_{22})+a_{12}y_0s^{q_2}}{s\Delta_A(s)}$$
is absolutely continuous and holomorphic in the open right half-plane, therefore, it does not have any poles in the open right half-plane. The function from the numerator is holomorphic on $\mathbb{C}\setminus\{s\in\mathbb{R},s\leq 0 \}$ and $\Delta_A(s)\ne 0$, for any $s\in\mathbb{C}$, $\Re(s)>0$.

Assuming that $\Delta_A(0)=0$, we obtain that $\det(A)=0$ and $\Delta_A(s)=s^{q_1+q_2}-a_{11}s^{q_2}-a_{22}s^{q_1}.$ Hence
	\begin{align*}
	\lim\limits_{s\rightarrow 0}sX(s)&=\lim\limits_{s\rightarrow 0}s\cdot \frac{x_0s^{q_1}(s^{q_2}-a_{22})+a_{12}y_0s^{q_2}}{s\Delta_A(s)}=\\
	&=\lim\limits_{s\rightarrow 0}\frac{x_0s^{q_1+q_2}-x_0a_{22}s^{q_1}+a_{12}y_0s^{q_2}}{s^{q_1+q_2}-a_{11}s^{q_2}-a_{22}s^{q_1}}=\\
	&=\lim\limits_{s\rightarrow 0}\frac{x_0s^{q_2}-x_0a_{22}+a_{12}y_0s^{q_2-q_1}}{s^{q_2}-a_{11}s^{q_2-q_1}-a_{22}}=\\
	&=\left\{
	\begin{array}{ll}
	x_0, & \textrm{if }a_{22}\neq 0 \\
	-\frac{a_{12}}{a_{11}}y_0, & \textrm{if }a_{22}=0
	\end{array}
	\right.\quad\neq 0,
	\end{align*}
This contradicts the Final Value Theorem for the Laplace transform $X(s)$ because $x(t)\rightarrow 0$ as $t\rightarrow\infty$. In conclusion, $\Delta_A(0)\neq 0$.

Let us now consider the solution $(x(t),y(t))$  of system (\ref{linearsys}) which satisfies the initial condition $(x(0),y(0))=\left(0,\frac{1}{a_{12}}\right)$. For $x(t)$ we obtain the Laplace transform $X(s)=\left[s^{1-q_2}\Delta_A(s)\right]^{-1}$. Assuming that $\Delta_A(s)$ has a root on the imaginary axis (but not at the origin), it follows that $X(s)$ has a pole on the imaginary axis, which implies that $x(t)$ has persistent oscillations, contradicting the convergence of $x(t)$ to the limit $0$, as $t\rightarrow\infty$.

Therefore, we obtain $\Delta_A(s)\neq 0$, for any $s\in\mathbb{C}$, $\Re(s)\geq 0$.
    
\emph{Part 1 - Sufficiency.} Let $(x(t),y(t))$ denote the solution of system (\ref{linearsys}) which satisfies the initial condition $(x(0),y(0))=(x_0,y_0)\in\mathbb{R}^2$. Assuming that all the roots of $\Delta_A(s)$ are in the open left half-plane, it follows that all the poles of the Laplace transforms functions $X(s)$ and $Y(s)$ given by (\ref{transfer.functions}) are either in the open left half-plane or at the origin, and $X(s)$ and $Y(s)$ have at most a single pole at the origin. A simple application of the Final Value Theorem of the Laplace transform \cite{Chen_Lundberg} yields
\begin{align*}
\lim_{t\rightarrow\infty}x(t)&=\lim_{s\rightarrow 0}sX(s)=\lim_{s\rightarrow 0}\frac{s^{q_1}(s^{q_2}-a_{22})x(0)+a_{12}s^{q_2}y(0)}{\Delta_A(s)}=0;\\
\lim_{t\rightarrow\infty}y(t)&=\lim_{s\rightarrow 0}sY(s)=\lim_{s\rightarrow 0}\frac{s^{q_2}(s^{q_1}-a_{11})y(0)+a_{21}s^{q_1}x(0)}{\Delta_A(s)}=0.
\end{align*}
The Laplace transform $X(s)$ is holomorphic in the right half-plane, except at the origin and has the asymptotic expansion
$$X(s)\sim \sum_{n=0}^\infty c_n s^{\lambda_n},\quad \textrm{as } s\rightarrow 0, $$
where $\lambda_0=q_1-1<\lambda_1<...<\lambda_n<...$. Based o Theorem 37.1 from \cite{Doetsch}, this leads to the following asymptotic expansion:
$$x(t)\sim \sum_{n=0}^\infty \frac{c_n}{\Gamma(-\lambda_n)}\frac{1}{t^{\lambda_n+1}} ,\quad \textrm{as } t\rightarrow \infty, $$
where $\Gamma$ represents the Gamma function with the convention
$$\frac{1}{\Gamma(-\lambda_n)}=0\quad \textrm{if }\lambda_n\in\mathbb{Z}_+.$$
As $\lambda_0+1=q_1$, it follows that $x(t)$ converges to $0$ as $t^{-q_1}$. Similarly, we also obtain that $y(t)$ converges to $0$ as $t^{-q_1}$. Combining the convergence results for the two components $x(t)$ and $y(t)$, it follows, based on Definition \ref{def.stability} that system (\ref{linearsys}) is $\mathcal{O}(t^{-q})$-globally asymptotically stable, where $q=\min\{q_1,q_2\}$.

\emph{Part 2.} Assume that $\det(A)\neq 0$, which is equivalent to $\Delta_A(0)\neq 0$. Consider the solution of $(x(t),y(t))$  of system (\ref{linearsys}) which satisfies the initial condition $(x(0),y(0))=(0,y_0)$, with an arbitrary $y_0\in\mathbb{R}^\star$. The Laplace transform of $x(t)$ is $X(s)=a_{12}y_0\left[s^{1-q_2}\Delta_A(s)\right]^{-1}$. Based on Proposition 3.1 from \cite{Bonnet_2002}, it follows that $\Delta_A(s)$ has a finite number of roots in $\mathbb{C}\setminus\mathbb{R}_{-}$, and in particular, in the open right half-plane. Obviously, the Laplace transform $X(s)$ is analytic in $\mathbb{C}\setminus\mathbb{R}_{-}$, except at the poles given by the roots of $\Delta_A(s)$.

If $\Delta_A(s)$ has at least one root in the open right half-plane, let us denote by $\rho>0$ the real part of a dominant pole of $X(s)$, i.e. $\rho=\max\{\Re(s):\Delta_A(s)=0\}$, and by $\nu\geq 1$ the largest order of a dominant pole. Following Theorem 35.1 from \cite{Doetsch}, we obtain that $|x(t)|$ is asymptotically equal to $k~t^{\nu-1} e^{\rho t}$ (with $k>0$) as $t\rightarrow \infty$. Hence, $x(t)$ is unbounded and therefore, system (\ref{linearsys}) is unstable.
\end{proof}

\begin{lemma}\label{lem.a.star}
Let $c>0$, $0<q_1<q_2\leq 1$, and consider the smooth parametric curve in the $(b,a)$-plane defined by
$$
\Gamma_{c,q1,q2}~:\quad
\begin{cases}
b=b(\omega)=\rho_1\omega^{q_2}-c\rho_2\omega^{-q_1}\\
a=a(\omega)=c\rho_1\omega^{-q_2}-\rho_2\omega^{q_1}
\end{cases},\quad \omega>0,
$$
where:
$$\rho_1=\frac{\sin\frac{q_1\pi}{2}}{\sin\frac{(q_2-q_1)\pi}{2}}> 0\quad,\quad \rho_2=\frac{\sin\frac{q_2\pi}{2}}{\sin\frac{(q_2-q_1)\pi}{2}}>1~.$$
The curve $\Gamma_{c,q1,q2}$ is the graph of a smooth, decreasing and convex bijective function $a^\star=a^\star_{c,q_1,q_2}:\mathbb{R}\rightarrow\mathbb{R}$ which satisfies the inequalities 
\begin{equation}\label{ineg.a.star}
\begin{cases}
a^\star(b)\leq (-b)^\frac{q_2}{q_1}c^{\left(1-\frac{q_2}{q_1}\right)} &\textrm{, if }b<0\\
a^\star(b)\leq -b^{\frac{q_1}{q_2}} &\textrm{, if }b\geq 0
\end{cases}
\end{equation}\end{lemma}

\begin{proof}
It is easy to see that for any $\omega>0$ one has:
\begin{align*}
b'(\omega)&=\rho_1q_2\omega^{q_2-1}+c\rho_2q_1\omega^{-q_1-1}>0\\
a'(\omega)&=-c\rho_1q_2\omega^{-q_2-1}-\rho_2q_1\omega^{q_1-1}<0
\end{align*}
Moreover, $a(0_+)=b(\infty)=\infty$ and $a(\infty)=b(0_+)=-\infty$, and hence, $\Gamma_{c,q_1,q_2}$ is the graph of a smooth decreasing bijective function $a^\star=a\circ b^{-1}:\mathbb{R}\rightarrow \mathbb{R}$ such that $a^\star(\pm\infty)=\mp\infty$. 

Furthermore, we compute: 
$$a''(\omega)b'(\omega)-a'(\omega)b''(\omega)=\omega^{-q_1-q_2-3} \left(\rho_1\rho_2q_1 q_2(q_2-q_1)(c^2+\omega^{2(q_1+q_2)})+2 c \omega^{q_1+q_2} (q_2^3 \rho_1^2-q_1^3 \rho_2^2)\right).$$
This expression is strictly positive, as $\frac{q_1}{q_2}\leq \frac{\rho_1}{\rho_2}\leq 1$ (which follows from the fact that the function $x\mapsto\frac{\sin x}{x}$ is decreasing on $(0,\pi)$). Therefore, as $b'(\omega)>0$, it follows by the chain rule that $\frac{d^2a^\star}{db^2}>0$, for any $b\in\mathbb{R}$, and hence, $a^\star$ is a convex function. 

Let us denote $\rho=\frac{\rho_1}{\rho_2}\in(0,1)$. We note that the root of the equation $a(\omega)=0$ is $\omega_a=(c\rho)^{\frac{1}{q_1+q_2}}$, while the root of $b(\omega)=0$ is $\omega_b=(c\rho^{-1})^{\frac{1}{q_1+q_2}}$. We have $\omega_a<\omega_b$ and hence, $a^\star(0)<0$.

We will first prove the inequality $a^\star(b)\leq -b^{\frac{q_1}{q_2}}$ for any $b\geq 0$, which is equivalent to $(-a(\omega))^{q_2}\geq b(\omega)^{q_1}$, for any $\omega\geq \omega_b$. Denoting $u=c\omega^{-q_1-q_2}\in(0,\rho]$, this further simplifies to proving the inequality
\begin{equation}\label{ineg.dem}
(\rho_1-u\rho_2)^{q_1}\leq (\rho_2-u\rho_1)^{q_2},\quad\forall~u\in(0,\rho].
\end{equation}
To prove inequality (\ref{ineg.dem}), we first apply Bernoulli's inequality and obtain
$$(\rho_1-u\rho_2)^{\frac{q_1}{q_2}}\leq 1+\frac{q_1}{q_2}(\rho_1-u\rho_2-1),\quad \forall~~u\in(0,\rho].$$
We now have to show that 
$$ 1+\frac{q_1}{q_2}(\rho_1-u\rho_2-1)\leq \rho_2-u\rho_1,\quad \forall~~u\in(0,\rho]$$
or equivalently 
$$u(\rho_1-r\rho_2)\leq \rho_2-1+r(1-\rho_1),\quad \forall~~u\in(0,\rho].$$
where $r=\frac{q_1}{q_2}\leq\rho$. This inequality holds if and only it is true for $u=\rho$, or equivalently: 
$$\sin\left(\frac{(q_1+q_2)\pi}{2}\right)\geq (1-r)\sin\left(\frac{q_2\pi}{2}\right),~\quad\forall~0<q_1<q_2\leq 1.$$
The above inequality follows from Jensen's inequality applied to the concave sine function: 
$$(1-r)\sin\left(\frac{q_2\pi}{2}\right)+r\sin(q_2\pi)\leq \sin\left((1-r)\frac{q_2\pi}{2}+r q_2\pi\right)=\sin\left(\frac{(q_1+q_2)\pi}{2}\right).$$

The other inequality, $a^\star(b)\leq (-b)^\frac{q_2}{q_1}c^{\left(1-\frac{q_2}{q_1}\right)}$ for any $b<0$, is equivalent to $a(\omega)^{q_1}\leq (-b(\omega))^{q_2}c^{q_1-q_2}$, for any $\omega<\omega_a$. Denoting $u=c^{-1}\omega^{q_1+q_2}\in(0,\rho)$, this further simplifies to inequality (\ref{ineg.dem}) which has been proved above. 
\end{proof}

With the aim of exploring the distribution of the roots of the characteristic function $\Delta_A(s)$ given above, the following result is given, which is a generalization of Proposition 2 from \cite{Brandibur_2016}.

\begin{proposition}\label{prop112}
	Consider the complex-valued function
	$$\Delta (s)=s^{q_1+q_2}+as^{q_2}+bs^{q_1}+c,$$
	where $0<q_1<q_2<1$, $s^{q_1}$ and $s^{q_2}$ represent the principal values (first branches) of the corresponding complex power functions and $a,b,c\in\mathbb{R}$.
	\begin{enumerate}
		\item If $c<0$, then $\Delta(s)$ has at least one positive real root.
		\item $\Delta(0)=0$ if and only if $c=0$.
		\item Assume that $c>0$.
		\begin{enumerate}
			\item If $a\geq 0$ and $b\geq 0$ then all the roots of $\Delta(s)$ satisfy $\Re(s)<0$.
			\item $\Delta(s)$ has a pair of pure imaginary roots if and only if
$$a=a^\star_{c,q_1,q_2}(b):=a^\star(b,c,q_1,q_2),$$
where $a^\star_{c,q_1,q_2}$ is the function given by Lemma \ref{lem.a.star}.
			\item If $s(a,b,c,q_1,q_2)$ is one of the roots of $\Delta(s)$ such that
			$$\Re(s(a^\star,b,c,q_1,q_2))=0,$$ where $a^\star=a^\star(b,c,q_1,q_2)$ defined at (b), the following transversality condition is satisfied:
			$$\frac{\partial \Re(s)}{\partial a}\Big|_{a=a^*}<0.$$
			\item All roots of $\Delta(s)$ are in the left half-plane if and only if $a>a^\star(b,c,q_1,q_2)$.
			\item $\Delta(s)$ has a pair of roots in the right half-plane if and only if $a<a^\star(b,c,q_1,q_2)$.
		\end{enumerate}
	\end{enumerate}
\end{proposition}

\begin{proof}

1. Since $\Delta(0)=c<0$ and $\Delta(\infty)=\infty$, due to the fact that $\Delta(s)$ is continuous on $(0,\infty)$, it results that it has at least one strictly positive real root.
	
\noindent 2. As $\Delta(0)=c$, it is easy to see that $\Delta(0)=0\Leftrightarrow c=0$.
	
\noindent 3.(a) Let $a\geq 0$, $b\geq 0$ and $c>0$. If $\Delta(s)$ had a root $s$ with $\Re(s)\geq 0$, then
	$$|\arg(s)|\leq\frac{\pi}{2}\Rightarrow
	|\arg(s^{q_i})|=q_i\cdot |\arg(s)|\leq\frac{q_i\pi}{2}<\frac{\pi}{2},~~ i\in\{1,2\}
	$$
	So $\Re(s^{q_1})>0$ and $\Re(s^{q_2})>0$. Moreover, we have
	$$s^{q_1+q_2}+as^{q_2}+bs^{q_1}+c=0\Leftrightarrow s^{q_1}=\frac{-as^{q_2}-c}{s^{q_2}+b}$$
	and we obtain
	\begin{align*} \Re(s^{q_1})&=\Re\Big(\frac{-as^{q_2}-c}{s^{q_2}+b}\Big)=\Re\left[\frac{(-as^{q_2}-c)(\bar{s}^{q_2}+b)}{|s^{q_2}+b|^2}\right]\\
	&=\frac{\Re\big[(-as^{q_2}-c)(\bar{s}^{q_2}+b)\big]}{|s^{q_2}+b|^2}\\
	&=\frac{\Re(-as^{q_2}-c)\Re(\bar{s}^{q_2}+b)-\Im(-as^{q_2}-c)\Im(\bar{s}^{q_2}+b)}{|s^{q_2}+b|^2}\\
	&=\frac{(-a\Re(s^{q_2})-c)(\Re(s^{q_2})+b)+a\Im(s^{q_2})(-\Im(s^{q_2}))}{|s^{q_2}+b|^2}\\
	&=\frac{-a|s|^{2q_2}-(ab+c)\Re(s^{q_2})-bc }{|s^{q_2}+b|^2}.
	\end{align*}
	Because $a\geq 0, b>0, c\geq 0$, we have that $-a|s|^{2q_2}-(ab+c)\Re(s^{q_2})-bc\leq 0$ and we obtain that $\Re(s^{q_1})\leq 0$, which contradicts $\Re(s^{q_1})>0$. In conclusion, the equation $\Delta(s)=0$ does not have any roots with $\Re(s)\geq 0$.
	
\noindent 3.(b) $\Delta(s)$ has a pair of pure imaginary roots if and only if there exist $\omega>0$ such that $\Delta(i\omega)=0$. As $i^q=\cos \frac{q\pi}{2}+i\sin\frac{q\pi}{2}$, taking the real and the imaginary parts of the equation $\Delta(i\omega)=0$, we obtain:
	\begin{equation}\label{sistemrealimaginar}
	\begin{cases}
	&\omega^{q_1+q_2}\cos\frac{(q_1+q_2)\pi}{2}+a\omega^{q_2}\cos\frac{q_2\pi}{2}+b\omega^{q_1}\cos \frac{q_1\pi}{2}+c=0\\	&
    \omega^{q_1+q_2}\sin\frac{(q_1+q_2)\pi}{2}+a\omega^{q_2}\sin\frac{q_2\pi}{2}+b\omega^{q_1}\sin \frac{q_1\pi}{2}=0
	\end{cases}
	\end{equation}
Solving this system for $a$ and $b$, using the notations from Lemma \ref{lem.a.star}, it follows that  $\Delta(s)$ has a pair of pure imaginary roots if and only if $(b,a)\in\Gamma_{c,q_1,q_2}$, or equivalently $a=a^\star_{c,q_1,q_2}(b):=a^\star(b,c,q_1,q_2)$.
	
\noindent 3.(c)
Let $s(a,b,c,q_1,q_2)$ be the root of $\Delta(s)$ such that $s(a^\star,b,c,q_1,q_2)=i\omega$, with $\omega>0$, where $a^\star=a^\star(b,c,q_1,q_2)$. Differentiating with respect to $a$ in the equation
	$$s^{q_1+q_2}+as^{q_2}+bs^{q_1}+c=0$$
we obtain
	$$(q_1+q_2)s^{q_1+q_2-1}\frac{\partial s}{\partial a}+s^{q_2}+aq_2s^{q_2-1}\frac{\partial s}{\partial a}+bq_1s^{q_1-1}\frac{\partial s}{\partial a}=0$$
	which is equivalent to
	$$\frac{\partial s}{\partial a}=\frac{-s^{q_2}}{(q_1+q_2)s^{q_1+q_2-1}+aq_2s^{q_2-1}+bq_1s^{q_1-1}}.$$
Therefore
	$$\frac{\partial \Re(s)}{\partial a}=\Re \left(\frac{\partial s}{\partial a}\right)=\Re\left( \frac{-s^{q_2}}{(q_1+q_2)s^{q_1+q_2-1}+aq_2s^{q_2-1}+bq_1s^{q_1-1}} \right).$$
which leads to
	$$\frac{\partial \Re(s)}{\partial a}\Big|_{a=a^*}=\Re\left( \frac{-\left(i\omega\right)^{q_2}}{(q_1+q_2)\left(i\omega\right)^{q_1+q_2-1}+a^*q_2\left(i\omega\right)^{q_2-1}+bq_1\left(i\omega\right)^{q_1-1}} \right).$$
	Denoting $P(\omega)=(q_1+q_2)\left(i\omega\right)^{q_1+q_2-1}+a^*q_2\left(i\omega\right)^{q_2-1}+bq_1\left(i\omega\right)^{q_1-1}$, we have
	\begin{equation}\label{real}
	\frac{\partial \Re(s)}{\partial a}\Big|_{a=a^*}=\Re \left( \frac{-\left(i\omega\right)^{q_2}}{P(\omega)} \right)=\omega^{q_2}\Re \left( \frac{-i^{q_2}\overline{P(\omega)}}{|P(\omega)|^2} \right)=-\frac{\omega^{q_2}}{|P(\omega)|^2}\cdot \Re\left(i^{q_2}\overline{P(\omega)}\right)
	\end{equation}
	In what follows, we compute $\Re\left(i^{q_2}\overline{P(\omega)}\right)$ and we obtain:
	\begin{align*}
	\Re\left(i^{q_2}\overline{P(\omega)}\right)&=\Re \left(\overline{i}^{q_2} P(\omega)\right)\\
	&=\Re \Big[ (q_1+q_2)\overline{i}^{q_2} i^{q_1+q_2-1}\omega^{q_1+q_2-1}+a^*q_2 \overline{i}^{q_2} i^{q_2-1}\omega^{q_2-1}+bq_1\overline{i}^{q_2} i^{q_1-1}\omega^{q_1-1} \Big]\\
	&=\omega^{q_1-1}\Re \Big[ (q_1+q_2)i^{q_1-1}\omega^{q_2}-a^\star q_2 i\omega^{q_2-q_1}+b q_1\overline{i}^{q_2}i^{q_1-1} \Big]\\
	&=\omega^{q_1-1} \Big[ (q_1+q_2)\omega^{q_2}\Re(i^{q_1-1})+b q_1\Re(\overline{i}^{q_2}i^{q_1-1}) \Big]\\
	&=\omega^{q_1-1}\Big[ (q_1+q_2)\omega^{q_2}\sin \frac{q_1\pi}{2}-b q_1\sin \frac{(q_2-q_1)\pi}{2} \Big]\\
	&=\omega^{q_1-1}\sin \frac{(q_2-q_1)\pi}{2}\Big[ (q_1+q_2)\omega^{q_2}\rho_1-q_1(\rho_1\omega^{q_2}-c\rho_2\omega^{-q_1})\Big]\\
    &=\omega^{q_1-1}\sin \frac{(q_2-q_1)\pi}{2}\Big[ q_2\rho_1\omega^{q_2}+c~q_1\rho_2\omega^{-q_1})\Big]>0
	\end{align*}
From (\ref{real}) it results that 
	$$\frac{\partial \Re(s)}{\partial a}\Big|_{a=a^*}<0.$$
	
	\noindent 3.(d,e) From the transversality condition obtained at 3.(c), we observe that $\Re(s)$ is decreasing in a neighborhood of $a^*$, so when $a$ decreases below the critical value $a^*=a^*(b,c,q_1,q_2)$, the pair of conjugated roots $(s,\overline{s})$ crosses the imaginary axis from the left half-plane to the right half-plane. Therefore, taking into consideration 3.(a), we obtain the conclusions.
\end{proof}	

Furthermore, sufficient stability conditions which do not depend on the fractional orders $q_1$ and $q_2$ will be obtained using the following:

\begin{proposition}\label{prop113}
	Let $c>0$ and $0<q_1<q_2<1$. For the complex-valued function $\Delta(s)$ defined in Proposition \ref{prop112} we have:
	\begin{enumerate}
		\item If $a+1>0$, $a+b>0$ and $b+c>0$, then all roots of $\Delta(s)$ are in the open left-half plane, regardless of $q_1$ and $q_2$.
		\item If $a+b+c+1\leq 0$ then the equation $\Delta(s)$ has at least one positive real root, regardless of $q_1$ and $q_2$.
	\end{enumerate}
\end{proposition}

\begin{proof}
\noindent 1. Assume that  $a+1>0$, $a+b>0$ and $b+c>0$. 

If $b\geq 0$, then $a>-1$ and $a>-b$ and inequality \ref{ineg.a.star} from Lemma \ref{lem.a.star} implies that
$$a> -\min\{b,1\}\geq -b^{\frac{q_1}{q_2}}\geq a^\star(b,c,q_1,q_2).$$

If $b<0$, we have $a>-b$ and $c>-b$ and inequality \ref{ineg.a.star} from Lemma \ref{lem.a.star} provides:
$$a> -b=(-b)^{\frac{q_2}{q_1}}(-b)^{1-\frac{q_2}{q_1}}\geq (-b)^{\frac{q_2}{q_1}}c^{1-\frac{q_2}{q_1}}\geq a^\star(b,c,q_1,q_2).$$

Hence, in both cases, based on Proposition \ref{prop112}.(d) it follows that all the roots of $\Delta(s)$ are in the open left half-plane.
 			
\noindent 2. As $\Delta(1)=1+a+b+c\leq 0$ and $\Delta(\infty)=\infty$, the function $\Delta(s)$ has at least one positive real root belonging to the interval $[1,\infty)$.
\end{proof}

Based on Theorem \ref{thm.lin.stab} and Propositions \ref{prop112} and \ref{prop113}, the following conditions for the stability of system (\ref{linearsys}) are obtained, with respect to its coefficients and the fractional orders $q_1,q_2$:

\begin{corollary}\label{cor.stab.lin} Consider the linear system (\ref{linearsys}) with $q_1,q_2\in(0,1)$, $q_1<q_2$, the fractional orders of the Caputo derivatives. Denoting $a=-a_{11}$, $b=-a_{22}$, $c=\det(A)$, the following hold:
	\begin{enumerate}
		\item If $c<0$, system (\ref{linearsys}) is unstable, regardless of the fractional orders $q_1$ and $q_2$.
		\item Assume that $c>0$ and $q_1,q_2$ are arbitrarily fixed. Consider $a^\star(b,c,q_1,q_2)$ given by Lemma \ref{lem.a.star}.
		\begin{itemize}
			\item[(a)] System (\ref{linearsys}) is $\mathcal{O}(t^{-q_1})$-asymptotically stable if and only if $a>a^\star(b,c,q_1,q_2)$.
			\item[(b)] If $a<a^\star(b,c,q_1,q_2)$ system (\ref{linearsys}) is unstable.
\end{itemize}
\item If $c>0$, the following sufficient conditions for the asymptotic stability and instability of system (\ref{linearsys}), independent of the fractional orders $q_1$, $q_2$ are obtained:
\begin{itemize}
            \item[(a)] If $a+1>0$, $a+b>0$ and $b+c>0$ system (\ref{linearsys}) is asymptotically stable, regardless of the fractional orders $q_1$ and $q_2$.
			
			\item[(b)] If $a+b+c+1\leq 0$  then system (\ref{linearsys}) is unstable, regardless of the fractional orders $q_1$ and $q_2$.
		\end{itemize}
	\end{enumerate}
\end{corollary}

\begin{ex}
Let us consider $c=4$, $q_1=0.4$ and $q_2=0.8$. The curve $\Gamma_{c,q1,q2}$ given by Lemma \ref{lem.a.star}, which represents the graph of the function $a^\star(b,c,q_1,q_2)$ is shown in Fig. \ref{fig.stab}. System (\ref{linearsys}) is $\mathcal{O}(t^{-0.4})$-asymptotically stable if and only if the parameters $(a,b)$ take values above the plotted curve $\Gamma_{c,q1,q2}$, based on Corollary \ref{cor.stab.lin}. On the other hand, if $(a,b)$ take values below the curve $\Gamma_{c,q1,q2}$, system (\ref{linearsys}) is unstable. For any values of $(q_1,q_2)$, if $(a,b)$ belong to the red shaded region shown in \ref{fig.stab}, system (\ref{linearsys})  is asymptotically stable. For any values of $(q_1,q_2)$, if $(a,b)$ belong to the blue shaded region shown in \ref{fig.stab}, system (\ref{linearsys})  is unstable. 

\begin{figure}[h]
	\centerline{
	\includegraphics[width=0.5\linewidth]{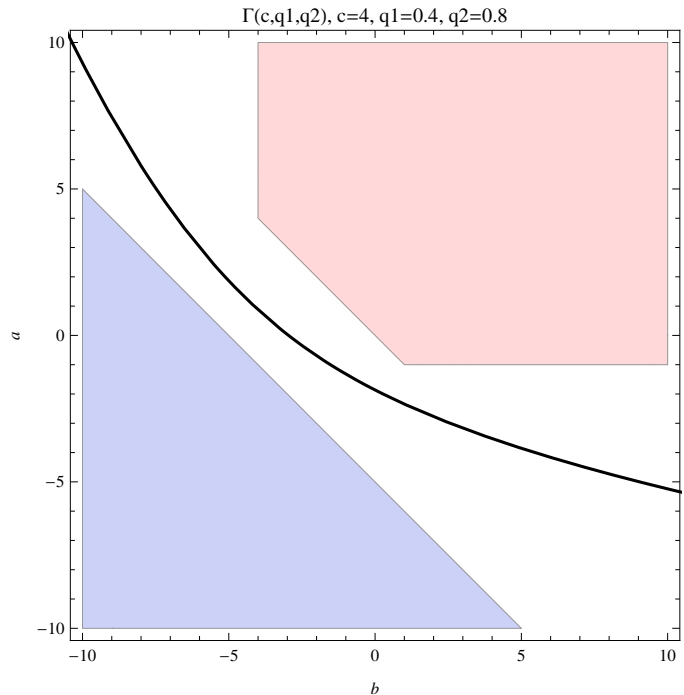}}
	\caption{Curve $\Gamma_{c,q1,q2}$ given by Lemma \ref{lem.a.star}, for fixed values of $c=4$, $q_1=0.4$ and $q_2=0.8$. The red/blue shaded regions represent the combination of parameters $(a,b)$ for which system (\ref{linearsys}) is asymptotically stable /unstable, regardless of the fractional orders $q_1$ and $q_2$.}
	\label{fig.stab}
\end{figure}

\end{ex}

\section{Investigation of a fractional-order FitzHugh-Nagumo model}

The FitzHugh-Nagumo neuronal model \cite{FitzHugh} is a simplification of the well-known Hodgkin-Huxley model, which describes a biological neuron's spiking behavior. In this paper, we consider an extension of the classical FitzHugh-Nagumo model, by replacing the integer-order derivatives by fractional-order Caputo derivatives:

\begin{equation}\label{FN1}
\left\{
\begin{array}{l}
^c\!D^{q_1}v(t)=v-\dfrac{v^3}{3}-w+I\\
^c\!D^{q_2}w(t)=r(v+c-dw)
\end{array}
\right.
\end{equation}
where $v$ represents the membrane potential, $w$ is a recovery variable, $I$ is an external excitation current and $0< q_1\leq q_2\leq 1$. A similar model has been investigated by means of numerical simulations in  \cite{Armanyos2016fractional}.

The second equation of system (\ref{FN1}) can be rewritten as follows:
$$^c\!D^{q_2}w(t)=rd\Big(\frac{1}{d}v+\frac{c}{d}-w\Big)=\phi(\alpha v+\beta -w)$$
where $\phi=rd\in (0,1)$, $\alpha=\dfrac{1}{d}>1$ and $\beta=\dfrac{c}{d}$. Therefore, system (\ref{FN1}) is equivalent to the following two-dimensional conductance-based model:
\begin{equation}\label{FN2}
\begin{cases}
 ^c\!D^{q_1}v(t)=I-I(v,w)\\
 ^c\!D^{q_2}w(t)=\phi(w_{\infty}(v)-w)
\end{cases}
\end{equation}
where $I(v,w)=w-v+\dfrac{v^3}{3}$ and $w_{\infty}(v)=\alpha v+\beta $ is a linear function.

The equilibrium states of the fractional-order neuronal model (\ref{FN2}) are the solutions of the following algebraic system
$$\begin{cases}
 I=I_\infty(v)\\
 w=w_{\infty}(v)
\end{cases}$$
where
\begin{align*}
I_{\infty}(v)=I(v,w_{\infty}(v))=w_{\infty}(v)-v+\frac{v^3}{3}=(\alpha-1)v+\dfrac{v^3}{3}+\beta.
\end{align*}
We observe that $I_\infty\in C^1$, $\lim\limits_{v\to -\infty}I_\infty(v)=-\infty$ and $\lim\limits_{v\to \infty}I_\infty(v)=\infty$. Moreover, 
$I'_{\infty}(v)=v^2+\alpha-1.$

As it is assumed that $\alpha>1$, it follows that $I'_{\infty}(v)=v^2+\alpha-1>0$, so the function $I_\infty$ is increasing and, as it is also continuous, it results that $I_\infty$ is bijective. Therefore, there exists a unique solution for the equation $I_\infty(v)=I$, which we denote by $v^*=v^*(I,\alpha,\beta)$.

For the investigation of the stability of equilibrium states, we consider the Jacobian matrix associated to system (\ref{FN2}) at an arbitrary equilibrium state $(v^*,w^*)=(v^*,w_{\infty}(v^*))$:
$$J=\begin{bmatrix}
1-(v^*)^2 & -1\\
\phi ~\alpha & -\phi
\end{bmatrix}$$
The characteristic equation at the equilibrium state $(v^*,w^*)$ is
\begin{equation}\label{car.eq}
s^{q_1+q_2}+a(v^*)s^{q_2}+b(v^*)s^{q_1}+c(v^*)=0
\end{equation}
where
\begin{align*}
& a(v^*)=-1+(v^*)^2\\
& b(v^*)=\phi >0 \\
& c(v^*)=\det(J)=\phi\cdot I'_{\infty}(v^*)>0.
\end{align*}

\begin{proposition}
	Any equilibrium state $(v^*,w^*)$ with $|v^*|>\sqrt{1-\phi}$ is asymptotically stable, regardless of the fractional orders $q_1$ and $q_2$.
\end{proposition}

\begin{proof}
As the function $I_\infty$ is increasing and
\begin{align*}    
& a(v^*)+1=(v^*)^2>0;\\
& a(v^*)+b(v^*)=-1+(v^*)^2+\phi>0;\\
& b(v^*)+c(v^*)=\phi+\phi\cdot I'_{\infty}(v^*)>0;
   \end{align*}
it follows from Corollary \ref{cor.stab.lin} 3.a, that  the equilibrium state $(v^*,w^*)$ is asymptotically stable, regardless of the fractional orders $q_1$ and $q_2$.
\end{proof}

The stability of any equilibrium state $(v^*,w^*)$ with $|v^*|\leq \sqrt{1-\phi}$ depends on the fractional orders $q_1$ and $q_2$ (see Figs. \ref{fig.branch} and \ref{fig.regiuni}).

\begin{figure}[h]
	\centerline{
	\includegraphics[width=0.6\linewidth]{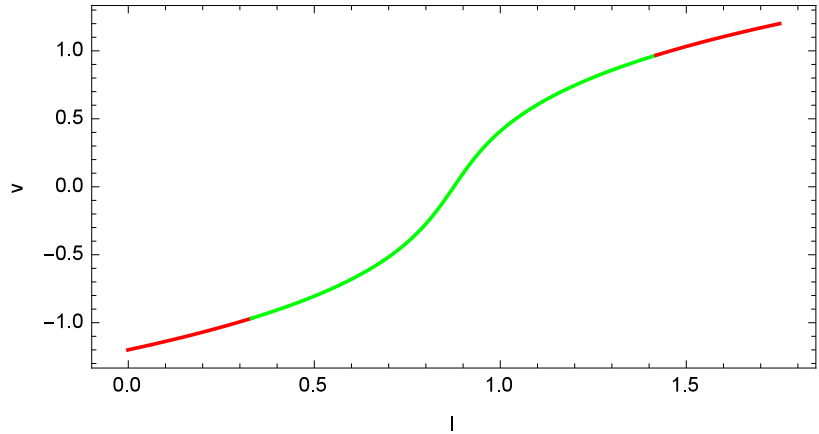}}
	\caption{Membrane potential ($v^*$) of the equilibrium states $(v^*,w^*)$ of system (\ref{FN1}) (with parameter values: $r = 0.08$, $c = 0.7$, $d = 0.8$) with respect to the external excitation current $I$ and their stability: red represents asymptotic stability, regardless of  the fractional orders $q_1$ and $q_2$; green represents equilibrium states whose stability depends on the fractional orders $q_1$ and $q_2$.}
	\label{fig.branch}
\end{figure}

\begin{figure}[h]
	\centering
	\includegraphics*[width=0.95\linewidth]{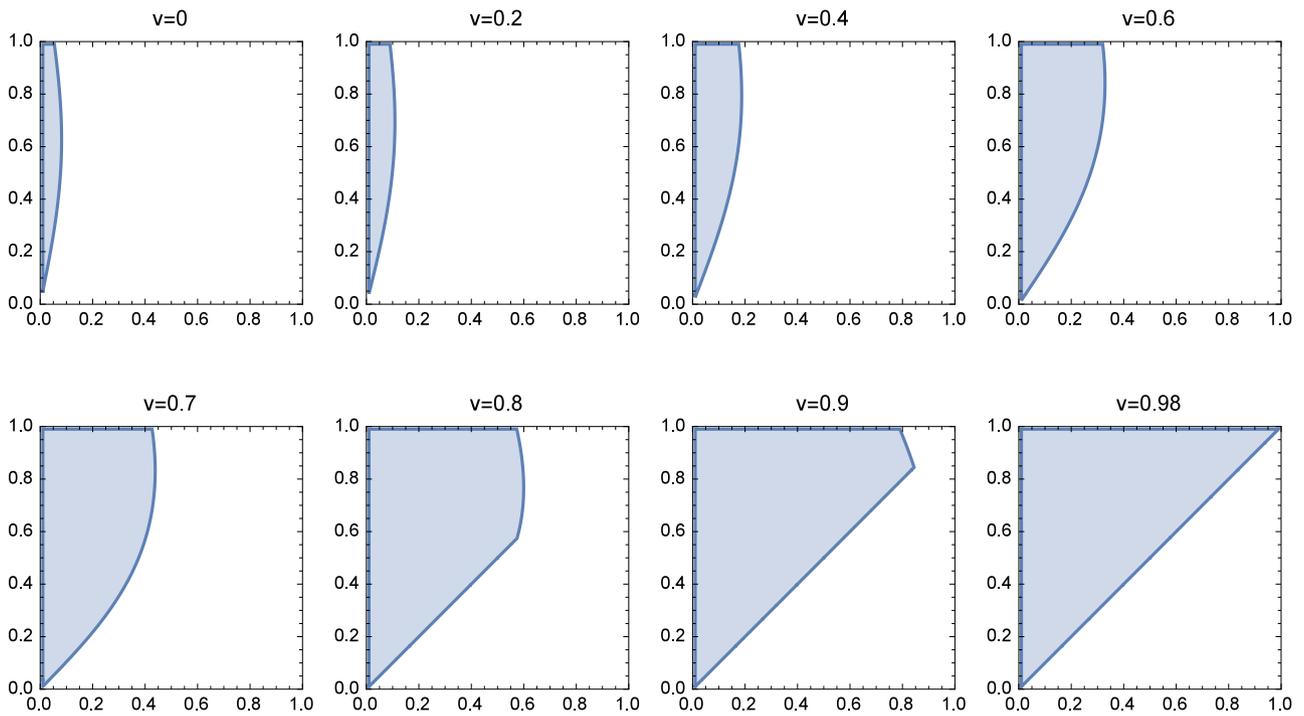}
	\caption{Stability regions (shaded) in the $(q_1,q_2)$-plane for equilibrium states $(v^*,w^*)$ of system (\ref{FN1}) (with parameter values: $r = 0.08$, $c = 0.7$, $d = 0.8$), with different values of the membrane potential $v^*$ satisfying the inequality $|v^*|\leq \sqrt{1-\phi}\approx 0.98$. In each case, the part of the blue curve strictly above the first bisector represents the Hopf bifurcation curve in the $(q_1,q_2)$-plane.}
	\label{fig.regiuni}
\end{figure}

Let us know consider an arbitrarily fixed equilibrium state $(v^*,w^*)$ of system (\ref{FN1}), such that $|v^*|\leq \sqrt{1-\phi}$. According to Proposition \ref{prop112}, at the critical values of the fractional orders $(q_1^\star,q_2^\star)$ defined implicitly by the equality $$a(v^\star)=a^\star(b(v^\star),c(v^\star),q_1,q_2),$$ a Hopf bifurcation is expected to occur (see Fig. \ref{fig.regiuni}).

Indeed, considering the following values for the system parameters: $r = 0.08$, $c = 0.7$, $d = 0.8$ and $I=1.24567$, the equilibrium state is $(v^*,w^*)=(0.8,1.875)$. In Fig. \ref{fig.spikes}, the evolution of the state variables is shown, considering an initial condition in a small neighborhood of the equilibrium point. For a fixed value $q_2=0.8$, the critical value of the fractional order $q_1$ for which a Hopf bifurcation occurs is $q_1^*=0.599$. Indeed, for $q_1=0.58$, asymptotically stable behavior is observed. For $q_1=0.63$, numerical simulations show quasi-periodic behavior, corresponding to the existence of a stable limit cycle. As $q_1$ is increased, the frequency of the oscillations increases. Numerical simulations suggest that fractional-order versions of the FitzHugh-Nagumo system provide a more realistic modeling of individual spikes than the corresponding integer-order counterpart (as seen in the last image from Fig. \ref{fig.spikes}).

\begin{figure}[htbp]
	\centering
	\includegraphics*[width=0.6\linewidth]{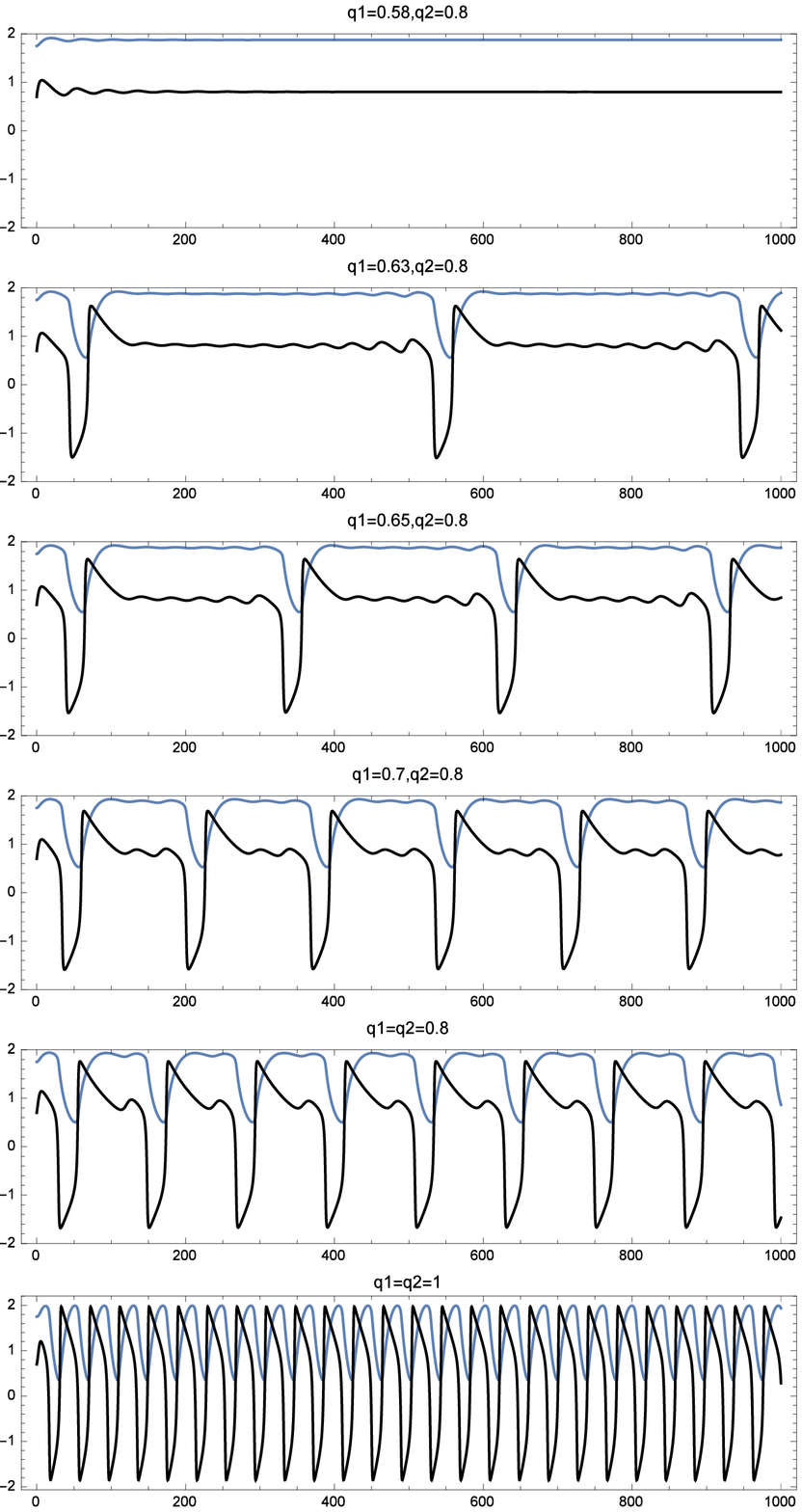}
	\caption{Evolution of the state variables of system (\ref{FN1}) (with parameter values: $r = 0.08$, $c = 0.7$, $d = 0.8$ and $I=1.24567$)   for different values of the fractional orders.}
	\label{fig.spikes}
\end{figure}

\section{Conclusions}

Necessary and sufficient conditions have been obtained for the asymptotic stability of a two-dimensional incommensurate order linear autonomous system with Caputo derivatives of different fractional orders. These results can be regarded as a generalization of the classical Routh-Hurwitz stability conditions. As an application, the stability properties of a fractional-order FitzHugh-Nagumo system have been explored. Numerical simulations are provided to exemplify the theoretical findings, additionally revealing the occurrence of Hopf bifurcations when critical values of the fractional orders are encountered.

%
%  Bibliography. Follow the usual conventions.
%

\bibliography{bibliografie,fractional_applications}

\end{document}